\documentclass[11pt]{amsart}
\usepackage{amsmath,amsfonts,amssymb,mathrsfs}
\title [$L^p$-estimates of the Boltzmann Equation]
{$L^p$-estimates of the Botlzmann Equation around a traveling local Maxwellian}
\author{Seok-Bae Yun}
\address{Department of Mathematical Sciences, KAIST (Korea Institute of Science and Technology)
373-1 Guseong-dong, Yuseong-gu Daejeon, 305-701, South Korea}
\email{sbyun01@gmail.com}

\begin{document}

\newtheorem{theorem}{Theorem}[section]
\newtheorem{lemma}{Lemma}[section]
\newtheorem{corollary}{Corollary}[section]
\newtheorem{proposition}{Proposition}[section]
\newtheorem{remark}{Remark}[section]
\newtheorem{definition}{Definition}[section]

\renewcommand{\theequation}{\thesection.\arabic{equation}}
\renewcommand{\thetheorem}{\thesection.\arabic{theorem}}
\renewcommand{\thelemma}{\thesection.\arabic{lemma}}
\newcommand{\bbr}{\mathbb R}
\newcommand{\bbt}{\mathbb T}
\newcommand{\bbs}{\mathbb S}
\newcommand{\bbz}{\mathbb Z}
\newcommand{\bn}{\bf n}

\def\charf {\mbox{{\text 1}\kern-.24em {\text l}}}

\begin{abstract}
In this paper, we are interested in the $L^p$-estimates
of the Boltzmann equation in the case that the distribution function
stays around a travelling local Maxwellian. For this, we divide
both sides of the Boltzmann equation by the velocity distribution function
with a fractional exponent and reformulate the Boltzmann equation into a
regularized one. This amounts to endowing additional integrability
on the collision kernel, which in turn enables us to apply simple H\"{o}lder type
inequalities. Our results cover the whole range of Lebesgue exponents: $0<p\leq \infty$.
\end{abstract} \maketitle


\section{Introduction}
In the kinetic theory of gases, it is postulated that all the relevant information
is encoded in a velocity distribution function $f(x,v,t)$ representing the number density
of particles located at position $x$ with velocity $v$  at time $t$.
For non-ionized monatomic rarefied gas, the time
evolution of $f$ is governed by the celebrated Boltzmann equation:
\begin{align}
\begin{aligned}
&\partial_tf + v \cdot \nabla_x f= Q(f,f),
\quad (x,v,t) \in \bbr^3 \times \bbr^3 \times \bbr_+. \\
\end{aligned}\label{main}
\end{align}
The left hand side of (\ref{main}) describes the free transport of non-interacting particles, whereas
the collision operator $Q(f,f)$ captures collisions or interaction between particles. It can be written
down explicitly as follows:
\begin{equation}
Q(f,f)(v) \equiv \frac{1}{\kappa} \int_{\bbr^3 \times \bbs_+^2} B(v-v_*, \omega) ( f^{\prime} f_*^{\prime} - ff_* )
d\omega dv_*. \label{Q}
\end{equation}
Here $\kappa$ is the Knudsen number which is the ratio between the
mean free path of molecules and the characteristic length of the
flow and $\displaystyle\bbs^2_+=\{\omega\in \bbs^2 ~|~ (v-v_*)\cdot\omega\geq 0\}$. For the simplicity of presentation, we adopt the
following handy notations:
\[
f^{\prime} \equiv f(x,v^{\prime},t), \quad f_*^{\prime}\equiv f(x,v_*^{\prime},t), \quad f \equiv f(x,v,t)
\quad \mbox{ and } \quad f_* \equiv f(x,v_*,t),
\]
where the pair $(v^{\prime}, v_*^{\prime})$ denotes the post-collisional
velocities which can be calculated explicitly from the pre-collisional pair
of velocities $(v,v_*)$ by
\begin{equation}
v^{\prime} = v - [(v -v_*) \cdot \omega]  \omega \quad
\mbox{ and } \quad v_*^{\prime} = v_* +  [(v - v_*) \cdot
\omega ] \omega. \label{coll}
\end{equation}
The collision kernel $B(v-v_*,\omega)$ is determined by types of interaction between gas particles.
For the precise form and relevant structural assumptions imposed on the collision kernel, see $({\mathcal A 1})$ below.
For more detailed survey of mathematical and physical results of the Boltzmann equation, we refer to
\cite{B-P-T, C, C-I-P, So1, So2, Vi}.

In this paper, we study the stability problem of the
Boltzmann equation in $L^p$ spaces when the velocity distribution function is bounded
from above and below by a travelling local Maxwellian:
\begin{equation}
a_{m}{\mathcal M}_{\alpha,\beta}(x,v)\leq f^{\sharp}(x,v,t)\leq a_{M}{\mathcal M}_{\alpha,\beta}(x,v),
\end{equation}
where $a_{m}$, $a_{M}$ denotes positive constants and ${\mathcal M}_{\alpha, \beta}(x,v)$ is a travelling local
Maxwellian solution:
\begin{equation}
{\mathcal M}_{\alpha,\beta}(x,v) \equiv
e^{-\alpha|x|^2-\beta|v|^2}  \quad \mbox{ for positive
constants }~~\alpha, \beta > 0. \label{local}
\end{equation}
For the stability problem of kinetic equations, $L^1$ space is the most natural setting in that
it corresponds to the total mass of the system. The study of stability in $L^1$ space for the
Boltzmann equation near vacuum was initiated by Ha
\cite{H1, H2} who introduced a nonlinear functional approach motivated by the stability
theory of hyperbolic conservation laws, and was studied extensively by Ha and his coworkers \cite{C-H, D-Y-Z, H-N-Y, H-Y}.
See also \cite{Ar, Lu}. 
It is then quite natural to ask whether the stability results in $L^1$ can be extended to general $L^p$ space.
Considering that the asymptotic behavior of
the Boltzmann equation in this regime is largely governed by the free transport equation:
\[
\frac{\partial f}{\partial t}+v\cdot \nabla f=0,
\]
for which the uniform $L^p$ stability estimate trivially holds, it is reasonable
to expect similar estimates to hold true for general $L^p$ spaces. In this vein, there have been several
results on the $L^p$-stability estimates of the Boltzmann equation near vacuum.
In \cite{H-Y-Y2}, Ha's nonlinear functional approach was extended to summational $L^p$ setting.
Then the Gronwall type argument also became available in \cite{H-L-Y} to obtain
weighted $L^p$-stability estimates.
Recently, Alonso et al. \cite{A-G} resolved the uniform $L^p$ stability problem for the Boltzmann
equation with soft potential in the affirmative.

The usual difficulty encountered in the study of $L^p$ type estimates of the collision operator is that
even the simple H\"{o}lder inequality cannot be directly applied due to the singularity of the collision kernel.
In \cite{H-L-Y}, this difficulty was overcome by introducing polynomial weights in the velocity fields. In this
paper, we attack this problem by dividing both sides of (\ref{main}) by $\frac{1}{\mu}f^{1-\mu}$ and
reformulating the the Boltzmann equation into the following form (See (\ref{divide})):
\[
\frac{\partial f^{\mu}}{\partial t}+v\cdot\nabla f^{\mu}=Q_{\mu}(f^{\mu},f^{\mu}).
\]
In this way, the reformulated collision operator $Q_{\mu}$ gains additional integrability, and we are now able
to apply H\"{o}lder type inequalities to obtain the following $L^p$-estimate:
\[
\|f^{\mu}\|_p\leq C_{\mu,p}\hspace{0.04cm}\|f^{\mu}\|_p,
\]
which, upon adjusting the value of $\mu$ and $p$ properly, leads to the main results. (See Theorem 1.1 and 1.2
below.) We mention that the parameter $\mu$ provides greater degree of freedom in determining the Lebesgue exponent,
which is a key element in obtaining $L^p$ estimates for $0<p<1$.
Before we state our assumptions and main results, we introduce the notion of
mild solutions.
\begin{definition}We say that a nonnegative function $f\in L^{\infty}(0,T; L^1(\bbr^3\times\bbr^3))$ is
a mild solution if it satisfies the mild form:
\begin{equation}
\displaystyle  f^{\sharp}(x,v,t) = f_0(x,v) + \int_0^t
Q^{\sharp}(f,f)(x,v,s)ds ,\qquad (x,v,t) \in \bbr^3 \times
\bbr^3 \times \bbr_+, \label{mild}
\end{equation}
where the operator $\sharp$ is defined by
\[
f^{\sharp}(x,v,t) \equiv f(x+ tv, v,t)
\]
\end{definition}
The global existence of mild solutions for the Boltzmann equation in infinite vacuum was first
established by Illner shinbrot \cite{I-S} in the case that the solution decays exponentially in phase
space by combining fixed point arguments with the celebrated Kaniel Shinbrot scheme \cite{K-S}.
Their result was then extended to more general settings including algebraically decaying
data by several authors \cite{Bel,To2,To3}. In \cite{P-T, To1} the smallness assumption imposed on
the upper travelling Maxwellian bound was replaced by a closedness condition to resolve the
Cauchy problem of the Boltzmann equation close to a local Maxwellian regime, which
is relevant to our case. We remark, however, that our stability analysis in this paper does not require
any closedness nor smallness restrictions on the solutions.
%
%
%
%
The main structural assumptions of this paper are as follows.
\begin{itemize}
\item $({\mathcal A1})$.~~The collision kernel satisfies an inverse
power potential and an angular cut-off assumption:
\[
B(v - v_*, \omega) = |v - v_*|^{\gamma} b_{\gamma}(\theta),\quad
-3 < \gamma \leq 1,
\]
and
\[
 \int_{\bbs^2_+}b_{\gamma}(\theta)d\omega= B_{\gamma} < \infty,
\]
where $\theta$ is the angle between $v -v_*$ and $\omega$.
\item $({\mathcal A2}).$ Mild solution $f$ satisfies
\[ \displaystyle a_{m}{\mathcal M}_{\alpha,\beta}(x,v)\leq
f^{\sharp}(x,v,t)  \leq a_{M}{\mathcal M}_{\alpha,\beta}(x,v), \quad
\mbox{ a.e.} \hspace{0.3cm} (x,v), \] for some strictly positive
constants $a_{M}$, $a_{m}$.
\end{itemize}
\begin{remark}
1. The existence of mild solution satisfying $(\mathcal{A}2)$ with
additional condition that $a_{M}-a_{m}$ is sufficiently small was established in \cite{P-T,To1}.
Recently, this result was extended to the classical solutions for soft potentials in \cite{A-G}.
\end{remark}
We are now in a position to state our main results. Below $G_p$ denotes constants
which depend on the Lebesgue exponent $p$, but not on $x,v$ and $t$.
%
%
\begin{theorem}\label{SoftPotentialTheorem}
Suppose that main assumption $({\mathcal A}1)$ holds with $-3<
\gamma\leq 1$ and let $f$ be a mild solution of
\eqref{main} satisfying $({\mathcal A}2)$
corresponding to an initial datum $f_0$. Then we have
\begin{eqnarray}
||f(t)||_{L^p} \leq G_p ||f_0||_{L^p}, \qquad0< p\leq \infty.\label{LpEstimateSoft}
\end{eqnarray}
\end{theorem}
\begin{remark}
1. Alonzo et al.\cite{A-G} has resolved $L^p$-stability problem of the Boltzmann equation
with soft potentials for spatially decaying solutions. Our result is weaker in the sense that we cannot consider the difference of
the two distribution functions: $f-\bar f$, but stronger in that it covers the hard potential case and
the whole range of exponent: $0<p\leq\infty$.\newline
2. 
We do not impose any smallness condition neither on
$a_{M}$ nor on $a_{M}-a_{m}$. Although the existence result was established only when the distribution functions
lie close to a local Maxwellian regime in the sense that $a_M-a_m$ is sufficiently small.
\end{remark}
The rest of this paper is organized as follows. In section 2,
we present several estimates which will be crucial for the later
sections.
Through section 3 to section 4, we prove our main
results. In the last section, we consider the stability problem of the
difference of two distribution functions.
%
%
\section{Preliminaries}
\subsection{Basic estimates}
In this part, we present several estimates to be used in later
sections. For the proof, we refer readers to \cite{H1, H-L-Y, H-Y, Lu}.
%
%
\begin{lemma}\label{Lemma1}
Let $x \in \bbr^3$,  $V \not = 0$ and $a>0$. Then we have
\begin{eqnarray*}
\int_0^{\infty} e^{-a|x + \tau V|^2}d\tau  &\leq&
\sqrt{\frac{\pi}{a}} \frac{1}{|V|}.
\end{eqnarray*}
\end{lemma}
%
%
\begin{lemma}\label{Lemma2}
For $-3<\gamma\leq 0$, we have
\begin{eqnarray*}
&& \displaystyle  \int_{\bbr^3 \times \bbs_+^2} B(v-v_*,\omega)
{\mathcal M}_{\alpha,\beta}(x+t(v-v_{*}),v_*) d\omega dv_*\cr
&& \hspace{2cm}\leq C(\gamma,\alpha,\beta)\cdot\frac{1}{(t+ 1)^{\gamma+3}}, \\
\end{eqnarray*}
where $\displaystyle C(\gamma,\alpha,\beta)=B_{\gamma}\Big[
\frac{2\pi}{\gamma+3}+\sqrt{\Big(\frac{\pi}{\alpha}\Big)^3}+\sqrt{\Big(\frac{\pi}{\beta}\Big)^3}~\Big].$
\end{lemma}
\section{The proof of theorem 1.1 $(-3<\gamma\leq 0)$}\setcounter{equation}{0}
Let $f$ be a mild solution of the Boltzmann equation
satisfying the structural assumption $(\mathcal{A}2)$. We then have from (\ref{mild})
\begin{eqnarray}
\frac{\partial f^{\sharp}}{\partial t}=
\frac{1}{\kappa}Q^{\sharp}(f,f)\leq \frac{1}{\kappa}Q^{+\sharp}(f,f).\label{BE}
\end{eqnarray}
We divide both sides of (\ref{BE}) by
$\frac{1}{\mu}(f^{\sharp}_{\varepsilon})^{1-\mu}$ $(0<\mu<1)$ to get
\begin{align}
\begin{aligned}\label{divide}
\frac{\partial (f^{\sharp})^{\mu}}{\partial t}
&\leq\frac{1}{\kappa}\frac{\mu}{(f^{\sharp})^{1-\mu}}Q^+(f,f)\cr
&= \frac{\mu}{\kappa}\int_{\bbr^3\times \bbs_+^2}B(v-v_*,\omega)\Big(\frac{f^{\prime\sharp}
f^{\prime\sharp}_*}{f^{\sharp}}\Big)^{1-\mu}\big(f^{\prime\sharp}f^{\prime\sharp}_{*}\big)^{\mu}d\omega dv_*.
\end{aligned}
\end{align}
We observe from the lower and upper bound estimate of $({\mathcal
A}2)$
\begin{eqnarray*}
\frac{f^{\prime\sharp}f^{\prime\sharp}_{*}}{f^{\sharp}}&\leq&
\frac{a_{M}^2e^{-\alpha|x-t(v-v^{\prime})|^2-\beta|v^{\prime}_*|^2}
e^{-\alpha|x-t(v-v^{\prime}_*)|^2-\beta|v^{\prime}_*|^2}}{a_{m}e^{\alpha|x|^2-\beta|v|^2}}\\
&=&\frac{a_{M}^2e^{-\alpha|x|^2+\beta|v|^2}
e^{-\alpha|x-t(v-v_*)|^2-\beta|v_*|^2}}{a_{m}e^{-\alpha|x|^2-\beta|v|^2}}\\
&=& \frac{a_{M}^2}{a_{m}}e^{-\alpha|x+t(v-v_*)|^2-\beta|v|^2}.
\end{eqnarray*}
We substitute the above estimate into (\ref{divide}) to obtain
\begin{align}
\begin{aligned}
\frac{\partial (f^{\sharp})^{\mu}}{\partial t}&\leq\mu e^{\alpha}\big(\frac{a_{M}^2}{a_{m}}\big)^{1-\mu}\int_{\bbr^3
\times \bbs_+^2}A_{\mu,\alpha,\beta}(v-v_*)b(\theta)\big(f^{\prime\sharp}
f^{\prime\sharp}_{*}\big)^{\mu}d\omega dv_*,
\end{aligned}\label{dividedBE}
\end{align}
where $A_{\mu,\alpha,\beta}(v-v_*)$ denotes the regularized collision kernel defined by
\[
A_{\mu,\alpha,\beta}\equiv A_{\mu,\alpha,\beta}(v-v_*)\equiv |v-v_*|^{\gamma}e^{-(1-\mu)(\alpha|x-t(v-v_*)|^2+\beta|v_*|^2)}.
\]
Note that $A_{\mu,\alpha,\beta}$ now is an integrable function, which is a crucial ingredient in estimating
the reformulated collision operator in $L^p$.
We then multiply $p(f^{\sharp})^{\mu (p-1)}$ to
(\ref{dividedBE}) and integrate over
$\bbr^3\times \bbr^3\times \bbr_+$ with respect to $(x,v,t)$ to
obtain
\begin{align}
\begin{aligned}\label{LPBE}
\hspace{-0.7cm}\|(f^{\sharp})^{\mu}(t)\|^{p}_{p}
&\leq\|(f^{\sharp}_{0})^{\mu}\|^{p}_{p}\cr
&+p\mu e^{\alpha}\Big(\frac{a_M^2}{a_m}\Big)^{1-\mu}\!\!\int^{\infty}_{0}\!\!\!\int_{\bbr^3\times \bbs_+^2}
\!\!A_{\mu,\alpha,\beta}b(\theta)
\big(f^{\prime\sharp}f^{\prime\sharp}_{*}\big)^{\mu}(f^{\sharp})^{\mu(p-1)}
d\omega dv_*dvdxdt.
\end{aligned}
\end{align}
For brevity, we put
\begin{eqnarray*}
{\mathcal N}_1(t)\equiv\!\!\!\int_{\bbr^3\times \bbs_+^2}\!\!A_{\mu,\alpha,\beta}(v-v_*)b(\theta)
\big(f^{\prime\sharp}f^{\prime\sharp}_{*}\big)^{\mu}\big(f^{\sharp})^{\mu(p-1)}
d\omega dxdvdv_*.
\end{eqnarray*}
%
%
%
%
\begin{lemma}\label{MainLemmaforSoftPotential} Let $\gamma \in (-2, 0]$. Then for $q\geq 1$, ${\mathcal N}_1$ satisfies the following
pointwise estimate:
\begin{align}
\begin{aligned}
{\mathcal N}_1(t)\leq \frac{C_{{\mathcal
N}_1}(a_M)^{\mu }}{(t+1)^{3+\gamma}}
\|(f^{\sharp})^{\mu}(t)\|^p_p,
\end{aligned}\label{I_II}
\end{align}
for some constant $C_{{\mathcal N}_1}=C_{{\mathcal N}_1}(\mu,\alpha,\beta)$.
\end{lemma}
\begin{proof}
We apply H\"{o}lder inequality to ${\mathcal N}_1$ to obtain
\begin{align}
\begin{aligned}\label{Holder}
{\mathcal N}_1&\leq
\int_{\bbs^2_+}b(\theta)\Big(\underbrace{\int_{\bbr^9}|v-v_*|^{\gamma}
e^{-\frac{(1-\mu)p}{p-1}(\alpha|x+t(v-v_*)|^2+\beta|v_*|^2)}
(f^{\sharp})^{\mu p}dvdv_*dx}_{{\mathcal N}_{1A}}\Big)^{\frac{p-1}{p}}\cr
&\times\Big(\underbrace{\int_{\bbr^9}|v-v_*|^{\gamma}
\big(f^{\prime\sharp})^{p\mu}
(f^{\prime\sharp}_{*}\big)^{p\mu}dvdv_*dx}_{{\mathcal N}_{1B}}\Big)^{\frac{1}{p}}d\omega.
\end{aligned}
\end{align}
(i) The estimate of ${\mathcal N}_{1A}$: We observe from Lemma \ref{Lemma2}
\begin{align}
\begin{aligned}\label{N_1A}
{\mathcal N}_{1A}&\equiv\int_{\bbr^9}|v-v_*|^{\gamma}e^{-\frac{(1-\mu)p}{p-1}(\alpha|x+t(v-v_*)|^2+\beta|v_*|^2)}
(f^{\sharp}(x,v,t))^{\mu p}dvdv_* dx\cr
&=\int_{\bbr^6}(f^{\sharp}(x,v,t))^{\mu p}
\Big(\int_{\bbr^3} |v-v_*|^{\gamma}e^{-\frac{(1-\mu)p}{p-1}(\alpha|x+t(v-v_*)|^2+\beta |v_*|^2)}dv_*\Big)dxdv\cr
&\leq\frac{1}{(t+1)^{3+\gamma}}\Big[~\frac{2\pi}{\gamma+3}+\sqrt{\Big(\frac{\pi(p-1)}{\alpha (1-\mu)p}\Big)^3}
+\sqrt{\Big(\frac{\pi(p-1)}{\beta (1-\mu)p}\Big)^3}~\Big]\|\big(f^{\sharp}(t)^{\mu}\big)\|^p_p\cr
&\equiv\frac{C_{{\mathcal N}_{1A}}}{(t+1)^{3+\gamma}}\|\big(f^{\sharp}(t)\big)^{\mu}\|^p_p.
\end{aligned}
\end{align}
(ii) The estimate of ${\mathcal N}_{1B}$: Applying a series of standard changes of
variables, we have
\begin{eqnarray*}
{\mathcal N}_{1B}&=&\int_{\bbr^9}|v-v_*|^{\gamma}(f(x+tv,v^{\prime}))^{p\mu}
(f(x+tv,v^{\prime}_*))^{p\mu}dxdv dv_*\cr
&=&\int_{\bbr^9}|v-v_*|^{\gamma}(f(x,v^{\prime}))^{p\mu}(f(x,v^{\prime}_*))^{p\mu}dxdv dv_*\cr
&=&\int_{\bbr^9}|v-v_*|^{\gamma}(f(x,v))^{p\mu}(f(x,v_*))^{p\mu}dv dv_*dx\\
&=&\int_{\bbr^9}|v-v_*|^{\gamma}(f^{\sharp}(x,v))^{p\mu}(f^{\sharp}(x+t(v-v_*),v_*))^{p\mu}dxdvdv_*.
\end{eqnarray*}
We then use Lemma \ref{Lemma2} to see
\begin{align}
\begin{aligned}\label{N_1B}
{\mathcal N}_{1B}&=\int_{\bbr^6}(f^{\sharp}(x,v))^{p\mu}
\Big(\int_{\bbr^3}|v-v_*|^{\gamma}(f^{\sharp}(x+t(v-v_*),v_*))^{p\mu}dv_*\Big)dv dx\cr
&\leq (a_M)^{p\mu}\int_{\bbr^6}\big((f^{\sharp}(x,v))^{p\mu}\big)
\Big(\int_{\bbr^3}|v-v_*|^{\gamma}e^{-p\mu(\alpha|x-t(v-v_*)|^2+\beta|v_*|^2)}dv_*\Big)dv dx\cr
&\leq \frac{(a_M)^{p\mu}}{(t+1)^{3+\gamma}}\Big[~\frac{2\pi}{\gamma+3}+\sqrt{\Big(\frac{\pi}{\alpha
p\mu}\Big)^3}+\sqrt{\Big(\frac{\pi}{\beta p\mu}\Big)^3}~\Big]\|(f^{\sharp})^{\mu}(t)\|^p_p\cr
&\equiv (a_M)^{p\mu}\frac{C_{{\mathcal N}_{1B}}}{(t+1)^{3+\gamma}}\|(f^{\sharp})^{\mu}(t)\|^p_p.
\end{aligned}
\end{align}
Substituting (\ref{N_1A}) and (\ref{N_1B}) into (\ref{Holder}), we
obtain
\begin{eqnarray*}
{\mathcal N}_{1}&\leq& (a_M)^{\mu}\big(C_{{\mathcal N}_{1A}}\big)^{\frac{p}{p-1}}\big(C_{{\mathcal N}_{1B}}\big)^{\frac{1}{p}}
(t+1)^{-(3+\gamma)(\frac{p-1}{p}+\frac{1}{p})}\int_{\bbs^2_+}b(\theta)\|(f^{\sharp})^{\mu}\|^p_pd\omega\\
&\leq& \big(C_{{\mathcal N}_{1A}}\big)^{\frac{p}{p-1}}\big(C_{{\mathcal
N}_{1B}}\big)^{\frac{1}{p}}\frac{(a_M)^{\mu}B_{\gamma}}{(t+1)^{3+\gamma}}\|(f^{\sharp})^{\mu}\|^p_p.
\end{eqnarray*}
We set
\[
C_{{\mathcal N}_1}(\alpha,\beta,
\mu)=(a_M)^{\mu}\big(C_{\mathcal{N}_{1A}}\big)^{\frac{p}{p-1}}\big(C_{\mathcal{N}_{1B}}\big)^{\frac{1}{p}}B_{\gamma}
\]
to complete the proof.
\end{proof}
We now substitute the estimate (\ref{I_II}) of Lemma \ref{MainLemmaforSoftPotential} into
(\ref{LPBE}) to obtain
\begin{align}
\begin{aligned}
\|(f^{\sharp})^{\mu}(t)\|^p_{p}&\leq\|(f^{\sharp}_{0})^{\mu}\|^p_{p}
+\mu pD_{\mu, p}\int^{t}_0\frac{1}{(t+1)^{3+\gamma}}\|(f^{\sharp})^{\mu}(t)\|^p_p~dt,
\end{aligned}
\end{align}
where
\[
\displaystyle
D_{\mu, p}=~a_{M}^{\mu}C_{\mathcal{N}_1}B_{\gamma}\Big(\frac{a_{M}^2}{a_{m}}\Big)^{1-\mu}.
\]
By Grownwall's lemma, this yields
\begin{eqnarray*}
\|f(t)\|^{\mu p}_{\mu p}\leq e^{2\mu p D_{\mu, p}}\|f_{0}\|^{\mu p}_{\mu p}
\end{eqnarray*}
or, equivalently,
\begin{eqnarray*}
\|f(t)\|_{\mu p}\leq e^{D_{\mu, p}}\|f_{0}\|_{\mu p}.
\end{eqnarray*}
We now adjust $\mu$ and $p$ to complete the proof. For this, assume we are given a Lebesque exponent $P\in (0,\infty)$.
We divide the argument into the following two cases:\newline
(i) $P\in [1,\infty)$: we fix $\mu$ between $0$ and $1$ and set $p=\frac{P}{\mu}$ to obtain
\begin{eqnarray*}
\|f(t)\|_{P}\leq e^{D_{\mu, p}}\|f_0\|_{P}.
\end{eqnarray*}
Letting $P \rightarrow \infty$, we get
\begin{eqnarray*}
\|f(t)\|_{\infty}\leq e^{D_{\mu, \infty}}\|f_0\|_{\infty}.
\end{eqnarray*}
Here $D_{\mu, \infty}$ denotes
\[
D_{\mu, \infty}=\lim_{P\rightarrow \infty}D_{\mu, \frac{P}{\mu}}<\infty.
\]
(ii) $P\in (0,1)$: we fix $p$ in $[1,\infty)$  and set $\mu=\frac{P}{p}$ to obtain
\begin{eqnarray*}
\|f(t)\|_{P}\leq e^{D_{\mu, p}}\|f_0\|_{P}.
\end{eqnarray*}
Note that in both cases $0<\mu<1$ and $1\leq p<\infty$ hold, which guarantee the relevance
of the preceding argument.
%
%
%
%
\section{The proof of theorem 1.1 $(-2<\gamma\leq 1)$}\setcounter{equation}{0}
If the intermolecular force is governed by  hard potentials $(0< \gamma \leq 1)$,
most of the crucial estimates in the previous sections are
not relevant anymore due to the unboundedness of the collision kernel at infinity.  We overcome this difficulty
by incorporating the idea of Cho and Yu \cite{C-Y} into the reformulated setting.
More precisely, we introduce a maximal distribution function $\sup_t f^{\sharp}$ and interchange the order of
integration between time and velocity, to resolve the singularity of the collision kernel at infinity.
We mention that the proof of this section is not restricted to the hard potential case
and can be applied to the soft potential case either for $-2<\gamma\leq 1$.
We again start from the following inequality:
\begin{eqnarray}
\frac{\partial (f^{\sharp})^{\mu}}{\partial t} \leq \mu
\big(\frac{a_{M}^2}{a_{m}}\big)^{1-\mu}\int_{\bbr^3 \times \bbs_+^2}
A_{\mu,\alpha,\beta}(v-v_*)b(\theta)\big(f^{\prime\sharp}
f^{\prime\sharp}_{*}\big)^{\mu}d\omega
dv_*.\label{Recall_BE2}
\end{eqnarray}
We integrate from $0$ to $t$ to obtain
\begin{align}
\begin{aligned}
(f^{\sharp}(t))^{\mu}&\leq(f_0^{\sharp})^{\mu}+\mu\big(\frac{a_M^2}{a_m}\big)^{1-\mu}\int^{t}_{0}\int_{\bbr^9
\times \bbs_+^2}A_{\mu,\alpha,\beta}b(\theta)\big(f^{\prime\sharp}
f_*^{\prime\sharp}\big)^{\mu}d\omega dv_*dt\\
&\leq(f_0^{\sharp})^{\mu}+\mu\big(\frac{a_M^2}{a_m}\big)^{1-\mu}\int^{\infty}_{0}\int_{\bbr^9
\times \bbs_+^2}A_{\mu,\alpha,\beta}b(\theta)\big(f^{\prime\sharp}
f_*^{\prime\sharp}\big)^{\mu}d\omega dv_*dt.\\
\end{aligned}
\end{align}
We then take the supremum in time to obtain
\begin{align}
\begin{aligned}\label{supf}
\sup_t(f^{\sharp})^{\mu}&\leq(f_0^{\sharp})^{\mu}+
\mu\big(\frac{a_{M}^2}{a_{m}}\big)^{1-\mu}\!\!\!\int^{\infty}_{0}
\!\!\int_{\bbr^3\times \bbs_+^2}\!\!\!\!A_{\mu,\alpha,\beta}b(\theta)(f^{\sharp}f_*^{\sharp})^{\mu}d\omega dv_*dt.
\end{aligned}
\end{align}
The reason why we do this will be clear in Lemma \ref{MainLemmaforHardPotentialSmall}.
We now take $L^p$ norm directly on both sides,
instead of multiplying $pf^{\sharp p-1}$ to both sides  of \eqref{supf} and integrating with
respect to $(x,v)$ as in the previous sections, to see
\begin{align}
\begin{aligned}\label{LpBE2}
\|\sup_t(f^{\sharp})^{\mu}\|_p&\leq\|(f_0^{\sharp})^{\mu}\|_p\cr
&+\mu\big(\frac{a_{M}^2}{a_{m}}\big)^{1-\mu}
\Big\|\int^{\infty}_{0}\int_{\bbr^3\times \bbs_+^2}A_{\mu,\alpha,\beta}b(\theta)
\big(f^{\prime\sharp}f_*^{\prime\sharp}\big)^{\mu}d\omega dv_*dt\Big\|_p\cr
&\equiv\|(f_0^{\sharp})^{\mu}\|_p+\mu\big(\frac{a_{M}^2}{a_{m}}\big)^{1-\mu} {\mathcal N}_2.
\end{aligned}
\end{align}
In the following lemma, we estimate ${\mathcal N}_2$. Note that ${\mathcal N}_2$ is bounded
by the $L^p$-norm of $\sup_t (f^{\sharp})$.
%
%
%
%
\begin{lemma}\label{MainLemmaforHardPotentialSmall} Let $\gamma\in (-2,1]$.
Then for $p\geq 1$ and $\mu \in (0,1)$, we have
\begin{eqnarray}\label{N_2}
{\mathcal N}_2 \leq C_{\mu,p}\|\big(\sup_{t}f^{\sharp}\big)^{\mu}\|_p
\end{eqnarray}
for some positive constant $C_{\mu,p}$
\end{lemma}
\begin{proof}
By H\"{o}lder inequality, we have
\begin{align}
\begin{aligned}\label{N_21}
{\mathcal N}_2&\leq\Big\|\int_{\bbs^2_+}b(\theta)\Big(\underbrace{\int^{\infty}_0\int_{\bbr^3}|v-v_*|^{\gamma}
 e^{-\frac{p(1-\mu)}{p-1}(\alpha|x+t(v-v_*)|^2+\beta|v_*|^2)}dv_*ds}_{{\mathcal N}_{2A}}\Big)^{\frac{p-1}{p}}\cr
 &\times\Big(\int^{\infty}_{0}\int_{\bbr^3}|v-v_*|^{\gamma}\big(f^{\prime\sharp})^{p\mu}
  (f_*^{\prime\sharp}\big)^{p\mu}dv_*ds\Big)^{\frac{1}{p}}d\omega\Big\|_{L^{p}(dx,dv)}.\cr
\end{aligned}
\end{align}
We use Lemma 2.1 and 2.2 to see
\begin{eqnarray*}
{\mathcal N}_{2A}&\equiv&\int^{\infty}_{0}\int_{\bbr^3}|v-v_*|^{\gamma}e^{-\frac{p(1-\mu)}{p-1}(\alpha|x+t(v-v_*)|^2+\beta|v_*|^2)}dtdv_*\\
&=&\int_{\bbr^3}|v-v_*|^{\gamma}e^{-\frac{p(1-\mu)}{p-1}\beta|v_*|^2}
\Big(\int^{\infty}_{0} e^{-\frac{p(1-\mu)}{p-1}\alpha|x+t(v-v_*)|^2}dt\Big)dv_*\cr
&\leq&\sqrt{\frac{\pi(p-1)}{\alpha p(1-\mu)}}\int_{\bbr^3}|v-v_*|^{\gamma-1}e^{-\frac{p(1-\mu)}{p-1}\beta|v_*|^2}dv_*\cr
&\leq&\sqrt{\frac{\pi(p-1)}{\alpha p(1-\mu)}}~\Big(\int_{|v_*|\leq 1}|v-v_*|^{\gamma-1}dv_*+
\int_{|v_*|>1}e^{-\frac{p(1-\mu)}{p-1}\beta|v_*|^2}dv_*\Big)\cr
&\leq&\sqrt{\frac{\pi(p-1)}{\alpha p(1-\mu)}}~\Big(~\frac{2\pi}{\gamma+2}+\sqrt{\Big(\frac{p-1}{\beta p(1-\mu)}\Big)^3}~\Big)\cr
&\equiv&C_{{\mathcal N}_{2A}}.
\end{eqnarray*}
Note that we performed integration in time first before the velocity integration.
We plug the above estimate of ${\mathcal N}_{2A}$ into (\ref{N_21}) to obtain
\begin{eqnarray*}
&&{\mathcal N}_{2}\leq({C_{{\mathcal
N}_{2A}}})^{\frac{p-1}{p}}\Big\|\Big(\int^{\infty}_0\int_{\bbr^3}|v-v_*|^{\gamma}
(f^{\sharp}(x-t(v-v^{\prime}),v^{\prime},t))^{p\mu}\cr
&&\hspace{1cm}\times(f^{\sharp}(x-t(v-v^{\prime}_*),v^{\prime}_*,t))^{p\mu}dv_*dt\Big)^{\frac{1}{p}}\Big\|_{L^p(dx,dv)}.
\end{eqnarray*}
Applying a series of changes of variables: $x+tv\rightarrow x$, $(v^{\prime}, v^{\prime}_*)\rightarrow(v, v_*)$
and $x\rightarrow x+tv$ gives
\begin{eqnarray*}
&&{\mathcal N}_{2}\leq(C_{{\mathcal N}_{2A}})^{\frac{p-1}{p}}\Big[\int^{\infty}_0\int_{R^9}|v-v_*|^{\gamma}(f^{\sharp}(x,v,t))^{p\mu}\cr
&&\hspace{1cm}\times(f^{\sharp}(x-t(v-v_*),v_*,t))^{p\mu}dx dvdv_*dt\Big]^{\frac{1}{p}}.
\end{eqnarray*}
We now introduce the maximal distribution $\sup_tf^{\sharp}(x,v)$ as follows
\begin{eqnarray*}
{\mathcal N}^p_{2}
&\leq&a_M^{p\mu}(C_{{\mathcal N}_{2A}})^{p-1}\int_{\bbr^6}\big(\sup_{t}(f^{\sharp}(x,v))^{p\mu}\big)\cr
&&\times\Big(\int_{\bbr^3}\int^{\infty}_{0}|v-v|^{\gamma}e^{-p\mu(\alpha|x-t(v-v_*)|^2+\beta|v_*|^2)}dtdv_*\Big)dv dx\cr
&\leq&a_M^{p\mu}(C_{{\mathcal N}_{2A}})^{p-1}\sqrt{\frac{\pi}{\alpha\mu
p}}\int_{\bbr^6}\big(\sup_{t}(f^{\sharp})^{p\mu}\big)
\Big(\int_{\bbr^3}|v-v|^{\gamma-1}e^{-p\mu\beta|v_*|^2}dv_*\Big)dv dx\cr
&\leq&a_M^{p\mu}(C_{{\mathcal N}_{2A}})^{p-1}\sqrt{\frac{\pi}{\alpha\mu
p}}\Big(~\frac{2\pi}{\gamma+2}+\sqrt{\Big(\frac{1}{\beta
p\mu}\Big)^3}~\Big) \|\sup_{t}(f^{\sharp})^{\mu}\|^p_p,
\end{eqnarray*}
where we used
\begin{eqnarray*}
f^{\sharp}(x,v,t)&\leq& \sup_t(f^{\sharp})(x,v)\quad\mbox{ and }\cr
f^{\sharp}(x,v_*,t)&\leq& a_Me^{-\alpha|x-t(v-v_*)|^2-\beta|v_*|^2}.
\end{eqnarray*}
Finally we put
\[
C_{\mu,p}\equiv  a_{M}^{\mu}(C_{{\mathcal N}_{2A}})^{\frac{p-1}{p}}
\Big[\sqrt{\frac{\pi}{\alpha\mu p}}\Big(~\frac{2\pi}{\gamma+2}+\sqrt{\Big(\frac{1}{\beta p\mu}\Big)^3}~\Big)\Big]^{\frac{1}{p}}
\]
to obtain the desired result.
\end{proof}
We now go back to the proof of the main theorem of this section. Substituting \eqref{N_2} into
\eqref{LpBE2} and recalling
\[
\|(f^{\sharp})^{\mu}\|_{p}=\|(f^{\sharp})\|^{\mu}_{\mu p},
\]
we have
\begin{eqnarray}\label{LpHard}
\|\sup_{t}(f^{\sharp})\|^{\mu}_{\mu p}\leq\|(f_0^{\sharp})\|^{\mu}_{\mu p}+
\bar C_{\mu,p}\|\sup_t(f^{\sharp})\|^{\mu}_{\mu p},
\end{eqnarray}
where
\begin{eqnarray*}
{\bar C}_{\mu,p}\equiv\mu a_{M}^{\mu}(C_{{\mathcal N}_{2B}})^{\frac{p-1}{p}}\Big(\frac{a_{M}^2}{a_{m}}\Big)^{1-\mu}
\Big[\sqrt{\frac{\pi}{\alpha\mu p}}\Big(~\frac{2\pi}{\gamma+2}+\sqrt{\Big(\frac{1}{\beta p\mu}\Big)^3}~\Big)\Big]^{\frac{1}{p}}.
\end{eqnarray*}
As in the previous section, we first fix $\mu p=P$  for a given Lebesgue exponent $0<P<\infty$.
We then observe that
\begin{eqnarray*}
{\bar C}_{\mu,p}\leq\mu\mathcal{O}(1)\Big[\sqrt{\frac{1}{P}}\Big(~1+\sqrt{\Big(\frac{1}{P}\Big)^3}~\Big)\Big]^{\frac{\mu}{P}},
\end{eqnarray*}
where we used the fact that
$(C_{{\mathcal N}_{2A}})^{\frac{p-1}{p}}$ is uniformly bounded for $p\geq1$, $0<\mu<1$,
and
\begin{eqnarray*}
a^{\mu}_M\Big(\frac{a_M^2}{a_m}\Big)^{1-\mu}=a_M\Big(\frac{a_M}{a_m}\Big)^{1-\mu}<\frac{a^2_M}{a_m}.
\end{eqnarray*}
Therefore, we can take $\mu$ sufficiently small (with $P$ fixed) such that ${\bar C}_{\mu,p}<1$,
which gives from (\ref{LpHard})
\begin{eqnarray*}
\|(f^{\sharp})(t)\|^{\mu}_{P}\leq\|\sup_t(f^{\sharp})\|^{\mu}_{P}\leq
\frac{1}{1-\bar{C}_{\mu,p}}\|(f_0^{\sharp})\|^{\mu}_P.
\end{eqnarray*}
This implies the desired result.
\section{On the stability of $f-\bar f$}\setcounter{equation}{0}
Let $f$, $\bar f$ be two mild solutions of (\ref{main}) which satisfy the upper bound estimate
(but not necessarily lower bound estimate) of the main assumption $({\mathcal A}2)$:
\[
({\mathcal A2})^{\prime}:\quad
0\leq f^{\sharp}(x,v,t),~\bar{f}^{\sharp}(x,v,t)\leq a_{M}{\mathcal M}_{\alpha,\beta}(x,v),\quad
\mbox{ a.e.} \hspace{0.3cm} (x,v),
\] for some strictly positive constant $a_{M}$.
 Since the difference $f-\bar f$
does not satisfies the lower bound estimate of $({\mathcal A}2)$ in general, the arguments given in section 3 and 4 are
not directly applicable to the difference of two distribution functions. One way to circumvent this problem is to consider
$g^{\sharp}(x,v,t)\equiv {\mathcal
M}^{-1}_{\alpha,\beta}f^{\sharp}(x,v,t)$ instead of $f^{\sharp}(x,v,t)$. Substituting this into
\eqref{main}, we obtain
\begin{eqnarray}
&&\frac{\partial g^{\sharp}}{\partial t}=  \int_{\bbr^3 \times
\bbs_+^2} A_{\alpha,\beta}(v-v_*, \omega) ( g^{\prime\sharp}
g_*^{\prime\sharp}
- g^{\sharp}g^{\sharp}_* )d\omega dv_*,\label{Eq_G1}\\
&&\frac{\partial \bar g^{\sharp}}{\partial t}=  \int_{\bbr^3 \times
\bbs_+^2} A_{\alpha,\beta}(v-v_*, \omega) ( \bar
g^{\prime\sharp}\bar g_*^{\prime\sharp} -\bar g^{\sharp}\bar
g^{\sharp}_* )d\omega dv_*,\label{Eq_G2}
\end{eqnarray}
where $A_{\alpha,\beta}$ denotes the regularized collision kernel as before:
\[
A_{\alpha,\beta}(v-v_*, \omega)=|v-v_*|^{\gamma}e^{-\alpha|x-(v-v_*)t|^2-\beta|v_*|^2}.
\]
We subtract (\ref{Eq_G2}) from
(\ref{Eq_G1}) and multiply $sgn(f^{\sharp}-\bar f^{\sharp})$ to both
sides to see
\begin{eqnarray*}
&&\frac{\partial G^{\sharp}}{\partial t}\leq  \int_{\bbr^3 \times
\bbs_+^2} A_{\alpha,\beta}(v-v_*, \omega) ( G^{\prime\sharp}
D_*^{\prime\sharp} +D^{\prime\sharp} G_*^{\prime\sharp} +
G^{\sharp} D^{\sharp}_* + D^{\sharp} G^{\sharp}_*)d\omega dv_*.
\end{eqnarray*}
where $G=|g-\bar g|$ and $D=|g+\bar g|$. Then the exactly same
arguments as in the previous sections yield
\begin{align}
\begin{aligned}
&\|G\|_p\leq C_p\|G_0\|_p,\hspace{0.5cm}(-3< \gamma \leq 1),
\end{aligned}\label{LpforG}
\end{align}
where $\theta=1$ for sufficiently small $a_M$. We now introduce the following
notation for simplicity.
\[
||f(t)||_{L^p_{\mathcal{M}}}\equiv\Big\{\int_{\bbr^6}\big(f^{\sharp}(x,v,t)\mathcal{M}^{-1}_{\alpha,\beta}\big)^pdxdv\Big\}^{\frac{1}{p}},
\]
then \eqref{LpforG} leads to the following theorems.
%
%
\begin{theorem}
Suppose that main assumption $({\mathcal A}1)$ holds with
$-3<\gamma \leq 1$. Let $f$ and $\bar f$ be mild solutions
satisfying  $({\mathcal A}2)^{\prime}$ corresponding to
initial data $f_0, \bar f_0$ respectively. Then we have
\[
||f(t)-\bar f(t)||_{L^p_{\mathcal{M}}}  \leq G_p ||f_0-\bar
f_0||_{L^p_{\mathcal{M}}}, \qquad t \geq 0.
\]
\end{theorem}
\noindent{\bf Acknowledgement} The author would like to thank anonymous reviewrs
for their valuable comments and suggestions.
This work was supported by BK21 Project.


\begin{thebibliography}{99}
\bibitem{A-G} R. Alonso, I. Gamba, Distributional and classical solutions to the Cauchy Boltzmann problem for
soft potentials with integrable angular cross section, J. Stat. Phys. {\bf 137} (2009) 1147-1165.

\bibitem{A-C-G} R. Alonso, E. Carneiro, I. Gamba, Convolution inequalities for the Boltzmann collision operator,
Commun. Math. Phys. {\bf 298} (2010) 293-232.

\bibitem{Ar} L. Arkeryd, Stability in $L^1$ for the spatially homogeneous Boltzmann equation.
Arch. Rational Mech. Anal. {\bf 103} (1988) 151-168.


\bibitem{B-P-T} N. Bellomo, A. Palczewski, G. Toscani, Mathematical topics in nonlinear kinetic theory,
World Scientific Publishing Co., Singapore, 1988.

\bibitem{Bel} N. Bellomo, G. Toscani, On the Cauchy problem for the nonlinear Boltzmann equation: global existence,
uniqueness and asymptotic behavior, J. Math, Phys. {\bf 26} (1985) 334-338.

\bibitem{C} C. Cercignani, The Boltzmann equation and its applications. Applied Mathematical Sciences,
67, Springer-Verlage, New York, 1987.

\bibitem{C-I-P} C. Cercignani, R. Illner, M. Pulvirenti, The mathematical theory of dilute gases,
 Applied Mathematical Sciences, 106, Springer-Verlag. New York, 1994.

\bibitem{C-H} M.-J. Chae, S.-Y. Ha, Stability estimates of the Boltzmann equation in a half space,
Quart. Appl. Math. 65 (2007) 653-682.

\bibitem{C-Y} Y.-K. Cho, B.-J. Yu, Uniform stability estimates for solutions and their gradients
to the Boltzmann equation: A unified approach. J. Differential Equations 245 (2008) 3615-3627.


\bibitem{D-Y-Z} R. Duan, T. Yang, C. Zhu, $L^1$ and BV-type stability of the Boltzmann equation with external forces,
J. Differential Equation 227 (2006) 1-28.

\bibitem{H1} S.-Y. Ha, $L^1$-stability of the Boltzmann equation for the hard shpere model. Arch. Rational Mech. Anal. 173 (2004) 279-296.

\bibitem{H2} S.-Y. Ha Noninear functionals of the Boltzmann equation and uniform stabilty estimates.
J. Differential Equation 215 178-205 (2005).

\bibitem{H-L-Y} S.-Y. Ha, H. Lee, S.-B. Yun, Uniform $L^p$-stability theory for the space-inhomogeneous
Boltzmann equation with external forces.  Discrete Contin. Dyn. Syst. 24 (2009) 115-143.

\bibitem{H-N-Y} S.-Y. Ha, S. Noh, S.-B. Yun, Global existence and stability of mild solutions to the Boltzmann
system for gas mixtures. Quart. Appl. Math. 65 (2007) 757-779.


\bibitem{H-Y-Y2} S.-Y. Ha, M. Yamazaki, and S.-B. Yun, Uniform $L^p$ stability estimate for the spatially inhomogeneous
Boltzmann equation near vacuum,  J. Hyperbolic Differ. Equ. 5 (2008) 713-739.



\bibitem{H-Y} S.-Y. Ha, Yun, S.-B. Yun, Uniform $L^1$-stability estimate of the Boltzmann equation
near a local Maxwellian, Physica D. 220 (2006) 79-97.


\bibitem{I-S} R. Illner, M. Shinbrot, Global existence for a rare gas in an infinite vacuum,
Commun. Math. Phys. 95 (1984) 217-226.

\bibitem{K-S} S. Kaniel, M. Shinbrot, The Boltzmann equation 1: Uniqueness and local existence,
Commun. Math. Phys. 58 (1978) 65-84.

\bibitem{Lu} X. Lu, Spatial decay solutions of the Boltzmann equation:
converse properties of long time limiting behavior, SIAM J. Math.
Anal. 30 (1999) 1151-1174.


\bibitem{P-T} A. Palczewski, G. Toscani,
Global solution of the Boltzmann equation for rigid spheres and initial data close to a local Maxwellian,
 J. Math. Phys. 30 (1989) 2445-2450.

\bibitem{So1} Y. Sone, Molecular gas dynamics: Theory, Techniques and Applications. Birkhauser, 2006.

\bibitem{So2} Y. Sone, Rarefied gas dynamics: Birkhauser, 2002.

\bibitem{To1} G. Toscani, Global solution of the initial value problem for the Boltzmann
equation near a local Maxwellian, Arch. Rational Mech. Anal. 102 (1988) 231-241

\bibitem{To2} G. Toscani, H-thoerem and asymptotic trend of the solution for a rarefied gas in a vacuum. Arch. Rational
Mech. Anal. 100 (1987) 1-12.

\bibitem{To3} G. Toscani, On the nonlinear Boltzmann equation in unbounded domains. Arch. Rational Mech. Anal. 95 (1986) 37-49.

\bibitem{Vi} C. Villani, A review of mathematical topics in collisional kinetic theory, in: Handbook of
Mathematical Fluid Mechanics, vol. I, North-Holland, Amsterdam, 2002, pp. 71-305.
\end{thebibliography}
\end{document}